\documentclass[11pt]{amsart}
\usepackage{amsmath}
\usepackage[all, cmtip]{xy}
\ifx\pdfoutput\undefined 
\usepackage{graphicx}
\else
\usepackage[pdftex]{graphicx}
\usepackage{epstopdf}
\fi
\usepackage{verbatim, manfnt}
\usepackage{tikz}

\newtheorem{theorem}{Theorem}[section]
\newtheorem{proposition}[theorem]{Proposition}
\newtheorem{lemma}[theorem]{Lemma}

\theoremstyle{definition}

\theoremstyle{remark}
\newtheorem{remark}[theorem]{Remark}
 
\numberwithin{equation}{section}

\newcommand{\h}{\hbar}

\newcommand{\bbR}{{\mathbb R}}

\newcommand{\calL}{{\mathcal L}}

\newcommand{\calR}{{\mathcal R}}
\newcommand{\calH}{{\mathcal H}}

\DeclareMathOperator{\tr}{Tr}

\newcommand{\norm}[1]{\left\Vert #1 \right\Vert}

\begin{document}


\title{Canonical forms for perturbations of the harmonic oscillator}
\author{V. Guillemin}\address{Department of Mathematics\\
Massachusetts Institute of Technology \\Cambridge, MA 02139 \\
USA}\thanks{V. Guillemin is supported in part by NSF grant
DMS-1005696.}\email{vwg@math.mit.edu}
\author{A. Uribe}\address{Department of Mathematics \\
University of Michigan \\ Ann Arbor \\ MI \\ 48109 \\ USA}
\thanks{A. Uribe is supported in part by NSF grant DMS-0805878.}
\email{uribe@umich.edu}
\author{Z. Wang}
\address{Department of Mathematics \\
University of Science and Technology of China \\ Hefei \\ Anhi \\ 230026 \\ China} \email{wangzuoq@ustc.edu.cn}

\date{}

\begin{abstract}
We consider a  class of perturbations of the 2D harmonic oscillator, and of
some other dynamical systems, which  
we show are isomorphic to a function of a toric system (a Birkhoff canonical form).   
We show that for such systems there exists a quantum normal form 
as well, which is determined by spectral data.
\end{abstract}

\maketitle
\tableofcontents

\section{Introduction}

This paper is a sequel to the paper \cite{GUW2}, in which we obtained inverse spectral 
results for the perturbed 2D harmonic oscillator by means of Moser averaging 
techniques.  In this paper we will show that these techniques can be also used to 
obtain classical and quantum Birkhoff canonical forms for the perturbed harmonic
oscillator, under certain conditions, 
and to show that these canonical forms are spectrally determined.  We
will also discuss other systems to which the present techniques apply.

In section 2 we will consider completely integrable hamiltonian systems on a symplectic 
four manifold $M$ for which one of the hamiltonians is the generator of a free 
circle action and for which the symplectic reductions of $M$ by this action are 
topologically two-spheres.  We will show that if the reduction of the second hamiltonian is
a perfect Morse function on all these two-spheres then one can conjugate this
integrable system to a very simple normal form.  (This result is in some sense
a special case of a much more general result of V\~u Ng\d{o}c and Pelayo
on semitoric systems (\cite{PS}), but since the proof is short and intuitive we will give it in its
entirety, which will allow us to sharpen it a little.)

In section 3 and 4 we will apply this result to the perturbed quantum harmonic oscillator and describe
an inductive scheme for passing from the classical Birkhoff canonical form for the
symbol of the perturbation to a quantum Birkhoff canonical form for the quantum
hamiltonian itself.

In \S 5 we will show that if $\lambda_i(\h,n)$, $1\leq i\leq n$, are the eigenvalues
of the perturbed harmonic oscillator in an interval
\[
[\h(n+1)-c\h^2, \h(n+1) + c\h^2]
\]
and $\mu_{n,\h}$ the spectral measure
\begin{equation}\label{sm}
\mu_{n,\h}(\rho) = \sum_i\rho\left(\frac{\lambda_i(n,\h)-\h(n+1)}{\h^2}\right)
\end{equation}
for $\rho\in C^\infty_0(\bbR)$, then as $\h$ tends to zero with $\h(n+1)$ fixed, 
(\ref{sm}) admits an asymptotic expansion in powers of $\h$ in which the terms in the 
expansion are expressible in terms of the quantum Birkhoff canonical form and also,
in turn, determine it.  (Hence in particular the quantum Birkhoff canonical form is itself
a spectral invariant.)

In \S 6 we will describe some inverse spectral results for perturbations of the
systems above for which the reduced averaged perturbations are not perfect Morse.
More explicitly let $P$ be the quantum harmonic oscillator and $Q_\h$ a perturbation
of $P$ by a semiclassical $\Psi$DO of degree $(-2)$ (e.g. $\h^3D_v + \h^2 V$
where $v$ is a vector field and $V$ a potential function) and consider spectral 
invariants of $P+Q_\h$ of the form
\begin{equation}\label{tr}
\tr\left(f(P) g(Q_\h^{\text{ave}})\right).
\end{equation}
We will then show that spectral invariants of this type are related to ``latitudinal"
Birkhoff canonical forms modulo support conditions on $f$ and $g$.  More precisely,
there is a simple canonical form for $P$ microlocally on regions $\sigma_P^{-1}(a,b)$ , provided 
$\sigma(Q_\h^{\text{ave}})$ does not have critical points on the region
\[
a< \sigma(Q_\h^{\text{ave}}) < b
\]
and its level sets are connected.  These canonical forms can be applied to the
study of (\ref{tr}) provided the support of $f$ is contained in $(a,b)$.

The techniques of this paper apply more or less verbatim to a number of other
quantum hamiltonians of Moser type, and in \S 7 we will describe two such examples.
The first is the standard Laplacian on the two sphere.
We will show that for perturbations of the form
\[
\Delta_{S^2} + Q_\h,
\]
where $Q_\h$ is a zeroth-order semiclassical pseudodifferential operator, there exist
a large class of $Q_\h$ for which the results of section 2 apply, and hence also
the theorems of \S 3 and 4 apply.  In particular we will show that if $D$ is a vector
field that is $C^2$ close to the vector field $x_1 \partial_{x_2} - x_2\partial_{x_1}$
these theorems can be applied to $Q_h = \frac{\h}{i} D$.
Our second example is also a Laplace operator on the two sphere:  the
Landau Laplacian associated with a constant magnetic field.
For the classical spherical Landau hamiltonian the periodic trajectories are 
circles of radius $\rho \leq \pi/2$ where $\rho$ depends on the strength of
the magnetic field.  When we perturb this system by a potential $V\in C^\infty(S^2)$,
the second hamiltonian of the completely integrable system,
in this case, is essentially the average of $V$ over such circles.  If $V$ is perfect
Morse, this averaged function will also be perfect Morse for $\rho$ small enough.
Hence the results of \S 2 (and consequently the results of \S 3 and 4) apply 
to perturbations of the Landau laplacian of the form $V$.

\section{From a semi-toric system to a toric system}


Let $(X,\omega)$ be a connected four-dimensional symplectic manifold and 
\[H,\, W: X\to\bbR\]
two Poisson-commuting hamiltonians.  We will make throughout the following assumption:
\begin{equation*}
\tag{A} H\ \text{generates a free } S^1\ \text{action on } X\ \text{and the reduced
spaces are 2-spheres.}
\end{equation*}

\newcommand{\slambda}{S^2_\lambda}
Let us denote by $I\subset\bbR$ an open interval in the image of $H$, and for any $\lambda\in I$
let 
\[
\slambda = H^{-1} (\lambda)/S^1
\]
be the reduced space, with symplectic form $\omega_\lambda$.  
We will denote by 
\[
W^\lambda: \slambda\to\bbR 
\] 
the hamiltonian obtained by reducing the function $W$ at $H=\lambda$.

We will let $M=H^{-1}(I)$ and make the following assumption:
\begin{equation*}
\tag{PM}  \forall \lambda\in I, \quad W^\lambda\ \text{is a perfect Morse function}.
\end{equation*}
(In \S 6 we discuss what can be said at this point without this assumption.)

Therefore for any $\lambda\in I$ the function $W^\lambda$ has exactly two 
critical points and they are non-degenerate:  A global maximum and a global minimum. 
We will denote
\[
a_\lambda = \min W^\lambda \quad\text{and}\quad b_\lambda =\max W^\lambda.
\]
It is clear that the image of the completely integrable system map
\begin{equation}
(H, W): M\to\bbR^2
\end{equation}
is the contractible region
\[
\calR = \{ (\lambda, w)\in\bbR^2\; ;\; \lambda\in I,\ a_\lambda\leq w \leq b_\lambda\}.
\]

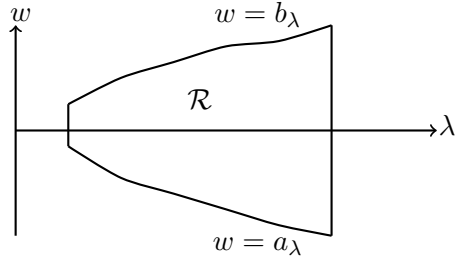
 \begin{center}
 \begin{tikzpicture}[scale=.7]
 
 \draw[thick, ->](0, 0)--(8, 0);
 \draw[thick, ->](0, -2)--(0, 2);
\draw[thick, smooth] plot coordinates
{(1,.5)(2,1)(3,1.3)(4,1.6)(5,1.7)(6,2)};
\draw[thick, smooth] plot coordinates
{(1,-.3)(2,-.9)(3,-1.2)(4,-1.5)(5,-1.8)(6,-2)};
 \draw[thick, smooth] plot coordinates
{(1,.5)(1,-.3)};
  \draw[thick, smooth] plot coordinates
{(6,2)(6,-2)};
\node at (3.5, .6) {$\calR$};
\node at (8.2, .1) {$\lambda$};
\node at (0.1, 2.2) {$w$};
\node at (4.6, 2.2) {$w=b_\lambda$};
\node at (4.6, -2.2) {$w=a_\lambda$};
\end{tikzpicture} 
\end{center}  

Our main goal in this section is to prove the following:
\begin{theorem}\label{2torusAction}
There exists a function $\Phi\in C^\infty(\calR)$ such that the map
\begin{equation}
(H, \Phi(H, W)): M\to\bbR^2
\end{equation}
is the moment map of a Hamiltonian $S^1\times S^1$ torus action on $M$.
\end{theorem}

We start by proving the following lemma:
\begin{lemma}
For any $\lambda\in I$ the push-forward measure $(W^\lambda)_* |\omega_\lambda|$,
where $|\omega_\lambda|$ is Liouville measure on $\slambda$,
is of the form
\[
(W^\lambda)_* |\omega_\lambda| = \rho_\lambda(w)\, dw
\]
where $\rho_\lambda \in C^\infty([a_\lambda, b_\lambda])$ is positive  everywhere,
and $dw$ is Lebesgue measure.  Moreover, $\rho$ is smooth as a function 
of $(\lambda, w)$.
\end{lemma}
\begin{proof}
We only need to check that $(W^\lambda)_*|\omega_\lambda|$ is of this form in a neighborhood of $a_{\lambda}$. Without loss of generality, we can take $a_\lambda=0$ and choose local coordinates on $S^2_\lambda$, depending smoothly on $\lambda$, so that near the minimum of $W^\lambda$, $W^\lambda(x, y)=x^2+y^2$  and $\omega_\lambda = \omega_\lambda(x, y) dx \wedge dy$ with $\omega_\lambda(0, 0)=1$. Let
\[
h_\lambda(w) = \int_{x^2+y^2 \le w} \omega_\lambda = \int_{x^2+y^2 \le w} \omega_\lambda(x, y)dxdy.
\]
Averaging the integrand with respect to the linear action of $SO(2)$ on $\mathbb R^2$ we can assume 
$\omega_\lambda (x,y)=\omega_\lambda(x^2+y^2)$, in which case
\[
h_\lambda(w) = \pi \int_{r^2 \le w} {\omega_\lambda}(r^2)dr^2,
\]
i.e. 
\[
\rho_\lambda(w) = \frac{dh_\lambda}{dw} = \pi \omega_\lambda(w).
\]
The positivity of $\rho_\lambda$ follows from the fact that $h_\lambda$ is strictly increasing. 
\end{proof}

\begin{proposition}\label{Varphi}
For any $\lambda\in I$ there is a strictly increasing 
function $\varphi_\lambda(w)$ such that 
\[
\varphi_\lambda\circ W^\lambda: \slambda\to\bbR
\]
is the hamiltonian of a circle action.  Moreover $\varphi_\lambda (w)$ is smooth
as a function of $(\lambda, w)$.
\end{proposition}
\begin{proof}
By the action-period relation, the condition on $\varphi_\lambda$ is equivalent to
the push-forward measure
$(\varphi_\lambda\circ W^\lambda)_* (|\omega_\lambda|)$ being Lebesgue measure
(restricted to its support).  But 
\[
\int_a^b (\varphi_\lambda\circ W^\lambda)_* (|\omega_\lambda|) = 
\int_{\varphi_\lambda^{-1}(a)}^{\varphi_\lambda^{-1}(b)} \rho_\lambda(w)\, dw
\]
which, by making the change of variables $w=\varphi_\lambda^{-1}(s)$ becomes
(omitting the parameter $\lambda$ for simplicity of notation)
\[
\int_a^b \rho (\varphi^{-1}(s))\, \frac{1}{\varphi'(\varphi^{-1}(s))}\, ds.
\]
Therefore it suffices to take $\varphi'(u) = \rho(u)$.
\end{proof}

With this proposition at hand we can prove Theorem \ref{2torusAction}.  
\begin{proof}[Proof of theorem \ref{2torusAction}]
Consider the Hamiltonian $F:M\to \bbR$ defined by
\[
F(x) = \varphi_{H(x)}(W(x)),
\]
and let $\Xi_F$ be its Hamilton vector field.  Since $H$ and $W$ Poisson commute
$\Xi_F$ is tangent to the level surfaces of $H$, and its integral curves project onto
integral curves of the Hamiltonians $\varphi_\lambda(W^\lambda)$.  Therefore
\[
\exp(2\pi \Xi_F): M\to M
\]
is a symplectic transformation that is ``equal to the identity modulo the circle action 
induced by $H$".  We now make this more precise.

Working locally, let $Y$ be the quotient $Y=H^{-1}(I)/S^1$ and let $h$ be the induced function on $Y$ so that 
\[\pi^* h = H.\]
 Take a local Darboux coordinate system $(y_1, y_2, h)$ on $Y$ so that the Poisson structure on $Y$ is $\partial_{y_1} \wedge \partial_{y_2}$. (For the existence of such coordinates system, c.f. \cite{LPV}).  Let $x_1=\pi^* y_1, x_2=\pi^* y_2$ and let $\theta$ be the local variable corresponding to the $S^1$ action. Then $(x_1, x_2, H, \theta)$ is a local coordinate system on $H^{-1}(I)$ such that 
\[
\{x_1, x_2\}=1, \quad \{x_1, H\}=0=\{x_2, H\}, \quad \{H, \theta\}=1.
\]
It follows that the Poisson bi-vector field is of the form
\[
\pi = \partial_{x_1}\wedge \partial_{x_2} + \partial_H \wedge \partial_\theta + c_1 \partial_{x_1}\wedge \partial_\theta + c_2 \partial_{x_2}\wedge \partial_\theta.
\]
Inverting the coefficient matrix, 
\[
\begin{pmatrix}
0 & 1 & 0 & c_1 \\
-1 & 0 & 0 & c_2 \\
0 & 0 & 0 & 1\\
-c_1 & -c_2 & -1 & 0
\end{pmatrix}^{-1}
=
\begin{pmatrix}
0 & -1 & c_2 & 0 \\
1 & 0 & -c_1 & 0 \\
-c_2 & c_1 & 0 & -1\\
0 & 0 & 1 & 0
\end{pmatrix},
\]
we see that the symplectic form in this coordinate system is of the form 
\[
\omega = dx_1 \wedge dx_2 + f_1 dx_1 \wedge dH  + f_2 dx_2 \wedge dH + dH \wedge d\theta. 
\]
Clearly the coefficients $f_1, f_2$ are only functions of $x_1, x_2, H$, i.e. they are independent of $\theta$. 

Consider now $\exp(2\pi \Xi_F): M \to M$. This is a symplectic transformation that preserves the coordinate 
functions $x_1, x_2$ and $H$.  Therefore in the previous coordinates
this transformation must take the form 
\[
(x_1, x_2, H, \theta) \mapsto (x_1, x_2, H, \theta+\chi)
\]
for some smooth function $\chi$ which a priori is a function of
$(x_1, x_2, H, \theta)$. Since this transformation preserves the symplectic 
structure, and the pull-back of $\omega$ equals
\[
 dx_1 \wedge dx_2 + f_1 dx_1 \wedge dH  + f_2 dx_2 \wedge dH + dH \wedge d(\theta+\chi), 
\]
 we must have $\chi=\chi(H)$.  
It follows that 
\[
\exp(-2\pi \chi \partial_\theta) \circ \exp(2\pi \Xi_F)=\mathrm{Id}.
\]
Finally we take $\Phi(H, W)=F-\chi(H) = \varphi(H, W)-\chi(H)$. This is the desired function.
\end{proof}

\section{Classical Birkhoff canonical forms}
In this section we will study a special semi-toric system that satisfies the assumptions in section 2. Let $X$ 
be the  cotangent space $T^*\mathbb R^2 \setminus \{0\}$ , with the standard symplectic form
\[
\omega = dx_1 \wedge d\xi_1 + dx_2 \wedge d\xi_2.
\]
As usual, we will let
\[
H_0 = \frac 12 (x_1^2+\xi_1^2+x_2^2+\xi_2^2)
\]
be the harmonic oscillator hamiltonian. We first consider a perturbation of $H_0$ of the form 
\[
H =H_0+\hbar^2 H_2 + O(\hbar^3)
\]
where $H_2 \in C^\infty(M)$. Let $\Xi_0$ be the hamiltonian vector field associated with $H_0$. Then $\Xi_0$ generates a circle action on $M$ which preserves $H_0$. 
By the Moser averaging argument (c.f. \cite{M}, \cite{Ur}),  
$H$ is symplectomorphic to 
\[
H_0+\hbar^2 H_2^{av}+ O(\hbar^3), 
\]  
where 
\[ 
H_2^{av} = \frac 1{2\pi} \int_0^{2\pi} (\exp t\Xi_0)^* H_2 dt
\] 
is the average of $H_2$ along the $\Xi_0$  orbits. In particular $\{H_0, H_2^{av}\}=0$. %
For the purposes of this paper,  we will be able to replace $H_2$ by its average.  Therefore, to simplify the notation, we will from  now on assume that $H_2 \in C^\infty(X)^{S^1}$. Let $I \subset (0, +\infty)$ be an open interval and let $M=H_0^{-1}(I)$. In what follows we will assume 
\begin{equation}\label{globalHyp}
\text{For any}\ \lambda\in I, \quad H_2^\lambda\ \text{is a perfect Morse function}
\end{equation}
where, as in last section, $H_2^\lambda: S_\lambda^2 \to \mathbb R$ is the hamiltonian obtained by reducing the 
function $H_2$ at $H_0=\lambda \in I$.

Under these assumptions, we get a completely integrable system map 
\[
(H_0, H_2): M \to \mathbb R^2. 
\]
Applying Theorem \ref{2torusAction}, one can find a smooth map $\Phi$ so that the map 
\[
\Psi = (H_0, \Phi(H_0, H_2))
\]
is the moment map of a Hamiltonian $S^1 \times S^1$ torus action whose image is a contractible region in $\mathbb R^2$. Moreover, from our construction we see that for any fixed $\lambda$, the function $\Phi(\lambda, \cdot)$ is invertible. We denote the inverse of this function
 by $\Phi_-(\lambda, \cdot)$. 
 
 Let 
 \begin{equation}\label{1dosc}
 H_0^i = \frac 12 (x_i^2+\xi_i^2),\quad i=1,2
 \end{equation}
 so that $H_0 = H_0^1+H_0^2$.
It is clear that there exists a function $\Phi_1$ so that the image of $(H_0, \Phi_1(H_0, H_0^1))$ 
is the same as the image of the moment map $\Psi$: 
 
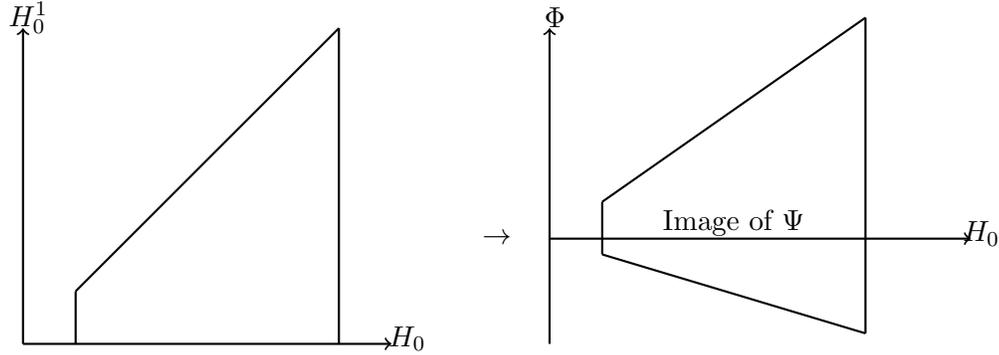
\begin{center}
 \begin{tikzpicture}[scale=.7]
 
  \draw[thick, ->](-10, -2)--(-3, -2);
 \draw[thick, ->](-10, -2)--(-10, 4);
\draw[thick, smooth] plot coordinates
{(-9,-1)(-8,0)(-7,1)(-6,2)(-5,3)(-4,4)};
 \draw[thick, smooth] plot coordinates
{(-9,-1)(-9,-2)};
  \draw[thick, smooth] plot coordinates
{(-4,4)(-4,-2)};
\node at (-2.7, -1.9) {$H_0$};
\node at (-9.9, 4.2) {$H_0^1$};

\node at (-1, 0){$\to$};
 
 \draw[thick, ->](0, 0)--(8, 0);
 \draw[thick, ->](0, -2)--(0, 4);
\draw[thick, smooth] plot coordinates
{(1,.7)(2,1.4)(3,2.1)(4,2.8)(5,3.5)(6,4.2)};
\draw[thick, smooth] plot coordinates
{(1,-.3)(2,-.6)(3,-.9)(4,-1.2)(5,-1.5)(6,-1.8)};
 \draw[thick, smooth] plot coordinates
{(1,.7)(1,-.3)};
  \draw[thick, smooth] plot coordinates
{(6,4.2)(6,-1.8)};
\node at (3.5, .3) {Image of $\Psi$};
\node at (8.2, .1) {$H_0$};
\node at (0.1, 4.2) {$\Phi$};
  \end{tikzpicture} 
\end{center}  
 
\noindent According to the non-compact version of Delzant's theorem due to Yael Karshon and Eugene Lerman (\cite{KL}), there is a symplectomorphism 
\[f: M \to M\]
such that 
\[
\Psi \circ f  = (H_0, \Phi_1(H_0, H_0^1)).
\] 
It follows from this that
\[
\Phi(H_0, H_2 \circ f) = \Phi(H_0, H_2) \circ f = \Phi_1(H_0, H_0^1), 
\]
i.e. 
\[
H_2 \circ f = \Phi_- (H_0, \Phi_1(H_0, H_0^1)).
\]
In other words, we proved
\begin{theorem}\label{CBC}
Under the assumption (\ref{globalHyp}),
there exists a symplectomorphism $f: M \to M$ which conjugates the Hamiltonian $H=H_0+\hbar^2 H_2 + O(\hbar^3)$ to the Birkhoff canonical form
\[
H_0+\hbar^2 F_2(H_0^1, H_0^2) + O(\hbar^3),
\]
where $F_2$ is a 2-variable smooth function, and $H_0^i = \frac 12 (x_i^2+\xi_i^2)$. 
\end{theorem}

\begin{remark}\label{F2Area}
From the proof we can see that for each fixed $\lambda$, the function $F_2(H_0^1, H_0^2)$ restricted to $H_0^1+H_0^2=\lambda$ is determined by the area function $A_\lambda(t) = \mathrm{Area}(H_2^\lambda<t)$, because the area function determines the function $\rho_\lambda$, which in turn determines the function $\varphi_\lambda$ in proposition \ref{Varphi}, and thus determines the function $\Phi(\lambda, \cdot)$ in theorem \ref{2torusAction}. 
\end{remark}

In what follows we will extend the theorem to higher order terms, i.e. we will consider perturbations of the form
\[
H =H_0+\hbar^2 H_2 + \hbar^3 H_3 + \hbar^4 H_4 +\cdots .
\]
Once again an application of the Moser averaging method allows us to assume without 
loss of generality  that $\{ H_0, H_j\}=0$ for all $j$.

\medskip
We begin by recalling that the reduced spaces $S^2_\lambda$ can be identified with the 
standard 2-sphere
of radius $\lambda/2$ in euclidean $\bbR^3$, with symplectic form $2/\lambda$ times the
area form induced by the euclidean structure.  (More intrinsically $\bbR^3$ is really the
Lie algebra su$(2)$, and $S^2_\lambda$ a coadjoint orbit of SU$(2)$.)
Under this identification the third ambient coordinate function, 
$X_3$,  is the reduction of the function
\[
\tilde X_3 = H_0^1-\lambda/2.
\]
$X_3$ generates a circle action on $S^2_\lambda$.

\begin{lemma}
Let $q$ be a perfect Morse function on $S^2_\lambda$ that is invariant under the above $S^1$
action (and therefore is a strictly monotone function of $X_3$).
Then for any $f \in C^\infty(S^2_\lambda)$, there exists a function 
$g \in C^\infty(S^2_\lambda)$ and an invariant 
function $f_0 \in C^{\infty}(S^2_\lambda)^{S^1}$ such that 
\[
\{q, g\} = f-f_0. 
\]
\end{lemma}
\begin{proof} 
We have that $q = h(X_3)$ where $h$ is a smooth one-variable function defined on 
the image of $X_3$ and with a nowhere vanishing derivative.  Explicitly, $h'(X_3)$
is a smooth, nowhere-vanishing function on the sphere.
Let $f_1$ be the average of the function $f/h'(X_3)$ with respect to the Hamiltonian
circle action, and let $g\in C^\infty(S^2_\lambda)$ be a function satisfying
\[
\{ X_3 , g\} = \frac{f}{h'(X_3)} - f_1.
\]
Then
\[
\{ q , g\} = \{h(X_3), g\} = h'(X_3)\{ X_3 , g\} = f - h'(X_3)f_1,
\]
i.e. $f_0 = h'(X_3)f_1$, which is obviously $S^1$ invariant.

\end{proof}

As a corollary, we get
\begin{proposition}
\label{AV}
Let $Q \in C^\infty(M)^{S^1}$ (i.e. invariant under the harmonic oscillator flow) satisfy 
the perfect Morse condition (\ref{globalHyp}) and be in Birkhoff canonical form, namely 
$Q=Q(H_0^1, H_0^2)$. Then for any $F \in C^\infty(M)^{S^1}$ there exists an invariant function 
$G \in C^\infty(M)$ and $F_0=F_0(H_0^1, H_0^2)$ such that 
\[
\{Q, G\} = F-F_0.
\]
\end{proposition}

As a consequence, we can now extend the Birkhoff canonical form in theorem \ref{CBC} to higher order terms:
\begin{theorem}\label{ClassicalBCF}
Under the assumption (\ref{globalHyp}), the perturbed harmonic oscillator Hamiltonian 
\[
H = H_0 + \hbar^2 H_2 + \hbar^3 H_3 + \cdots
\]
can be conjugated by a symplectomorphism on $M$ to the Birkhoff canonical form 
\[
H_0+\hbar^2 F_2(H_0^1, H_0^2)+ \hbar^3 F_3(H_0^1, H_0^2)+  \cdots,
\]
where $F_2$, $F_3$ etc are two variable smooth functions.
\end{theorem}
\begin{proof}
We have seen that $H$ can be conjugated to 
\[
H_2+\hbar^2 F_2(H_0^1, H_0^2) + \hbar^3 \widetilde H_3 + O(\hbar^4). 
\]
We apply Theorem \ref{AV} with $Q=F_2$ and $F=\widetilde H_3$.  Then there exists an invariant function $G \in C^\infty(M)$ and $F_3 =F_3(H_0^1, H_0^2)$ so that 
\begin{equation}\label{coheq}
\{F_2, G\}=\widetilde H_3 - F_3.
\end{equation}
It follows that
\[
\exp(\hbar \Xi_G)^* (H_2 + \hbar^2 F_2+ \hbar^3 \widetilde H_3 + O(\hbar^4) )= H_2 + \hbar^2 F_2 + \hbar^3 F_3 + \hbar^4 \widetilde H_4 + O(\hbar^5). 
\]
Repeating this argument iteratively, we can convert each term into a function that 
depends only on $H_0^1$ and $H_0^2$ (at each stage one solves an equation of the
form (\ref{coheq}) involving $F_2$, and therefore Theorem \ref{AV} can be applied).
\end{proof}

\section{Quantum Birkhoff canonical forms}

In this section we will apply Theorem \ref{ClassicalBCF} to obtain a Birkhoff canonical form theorem for perturbations of the two dimensional  harmonic oscillator of the form
\[
\hat H = \hat H_0 + \hbar^2 \hat H_2+ \hbar^3 \hat H_3+ \cdots,
\]
where 
\[
\hat H_0 = -\frac { \hbar^2}2  (\frac{\partial^2}{\partial x_1^2} + \frac{\partial^2}{\partial x_2^2}) + \frac 12(x_1^2+x_2^2)
\]
and $\hat H_i (i \ge 2)$  are zeroth order semiclassical pseudo-differential operators. 
By the quantum version of Moser averaging we can conjugate such a perturbation 
by a unitary operator to one where the perturbation commutes with $\hat H_0$.  
Therefore without loss of generality we will assume that
$[\hat H_0 , \hat H_j] = 0$ for all $j\geq 2$. 
 
\begin{theorem}\label{QBC}
Under the assumption (\ref{globalHyp}) on the principal symbol of $\hat{H}_2$,
there exists a FIO, $\mathcal F$, that conjugates $\hat H$ into an operator of the form 
\[
\hat H_0 + \hbar^2 G_2 (\hat H_0^1, \hat H_0^2) + \hbar^3 G_3(\hat H_0^1,\hat H_0^2) + \cdots + \h^2\hat R,
\]
where $\hat H_0^i=-\frac{\hbar^2}2 \frac{\partial^2}{\partial x_i^2}+\frac{x_i^2}2$ for $i=1, 2$, the $G_j$'s are two-variable smooth functions for $j \ge 2$, and $\hat R$ is a pseudodifferential operator whose microsupport is disjoint from $M$.
\end{theorem}
\begin{proof}
Let $\hat F$ be a Fourier integral operator such that:
\begin{enumerate}
\item Its canonical relation intersected with $M \times M$ coincides with the 
canonical transformation constructed above, that conjugates $H_0+\h^2 H_2$ to
$H_0 + \h^2F_2(H_0^1, H_0^2)$ modulo higher-order terms,
\item Both $\hat F\hat F^*$ and $\hat F^*\hat F$ are microlocally equal to the identity on $M$.
\end{enumerate}
Then conjugating $\hat H_0 + \hbar^2 \hat H_2+\cdots$ by $\hat F$ gives us 
a pseudodifferential operator of the form
\[
\hat H_0 + \hbar^2  G_2 (\hat H_0^1, \hat H_0^2) + \h^3 \hat G_3 +h^2 \hat R_3
\]
where $\hat G_3$ and $\hat R_3$ are pseudodifferential operators of order zero and 
the microsupport of $R_3$ is disjoint from $M$.

Using Theorem \ref{ClassicalBCF} we can continue in this form indefinitely,
exactly as in \cite{GuW}.

\end{proof}

\section{Relation between the band invariants and the BCF}

We first show that the classical Birkhoff canonical form function $F_2$ in Theorem \ref{CBC} is determined by the semiclassical 
spectrum of the corresponding perturbed harmonic oscillator. 

\begin{proposition} For any given $\lambda>0$ the function $F_2(\lambda, \cdot)$ 
can be recovered from the semiclassical
spectrum of the perturbed harmonic oscillator $\hat H_0 + \h^2 \hat H_2+O(\hbar^3)$.
\end{proposition}
\begin{proof}
Let us consider the sequence of values of $\hbar$ given
by 
\[
\lambda = j\hbar, \quad j=1,\,2,\ldots .
\]
Along this sequence, the spectrum of $\hat H_0 + \h^2 \hat H_2 + \cdots$ in a neighborhood 
of $\lambda$ of size $O(\hbar)$ consists, for $\h$ small enough, of eigenvalues  of the form
$\lambda+\hbar^2\mu_{j,\ell, \hbar}$, and therefore 
the ``eigenvalue shifts" $\mu_{j,\ell, \hbar}$ are spectrally determined (for the above 
values of $\hbar$).  
By the Szeg\"o limit theorem for the perturbations considered here (c.f. \S 3 in \cite{GUW2}),
the asymptotic distribution of the $\mu_{j,\ell, \hbar}$ is given in the weak sense 
by the push-forward measure $(H_2^\lambda)_*(\omega_\lambda)$, where $H_2^\lambda$
is the reduction of $H_2$ to $S_2^\lambda$, and $\omega_\lambda$ the symplectic form of
$S^2_\lambda$.  
This allows us to determine the function of $t$
\[
\mathrm{Area}(H_2^\lambda< t),
\]
which determines the function $F_2$ as explained in remark \ref{F2Area}.
\end{proof}

Next we will prove a  version of this for the quantum canonical form. 
We will consider the operator 
\[
\hat H = \hat H_0 +\hbar^2 \hat H_2 + \hbar^3 \hat H_3 + \cdots.
\]
According to theorem \ref{QBC}, the  functions $G_j$ appearing in the quantum Birkhoff 
normal form  determine the spectrum of $\hat H$ modulo $O(\hbar^\infty)$. 
Conversely, we can prove:
\begin{theorem}
The spectrum of $\hat H$ determines the quantum Birkhoff normal form 
functions $G_2, G_3$ etc. 
\end{theorem}
\begin{proof}
Since the micorsupport of $\hat R$ is away from $M$, the eigenvalues of $\hat H$ in
the interval $I$ (defining the manifold $M$ where the ``perfect Morse condition" 
(\ref{globalHyp}) is assumed to hold) are of the form
\[
\hbar k_1+\hbar k_2+ \hbar^2 G_2(\hbar k_1, \hbar k_2) +\hbar^3 G_3(\hbar k_1, \hbar k_2) + \cdots + O(\hbar^\infty),
\]
where $k_1, k_2 \in  \mathbb Z_+$. 
Consider the joint spectral measure of $\hat H_0$ and $\hat H$, defined by 
\[
\mu(f) = \mathrm{Tr}(f(\hat H_0, \hat H))=\sum f(\hbar k_1+\hbar k_2,G_\h(\hbar k_1, \hbar k_2)), 
\]
where $f \in C_0^\infty(\mathbb R^2)$ is any compactly-supported smooth function, and 
\[
G_\h(s, t)=s+t+\hbar^2  G_2(s, t)+ \hbar^3 G_3(s, t)+ \cdots
\] 
is a Borel summation of the quantum Birkhoff normal form. 
According to the Euler-Maclauin formula proven by S. Sternberg the the first author in (\cite{GS}), 
as $\hbar \to 0$ we have the following asymptotics modulo $O(\hbar^\infty)$:
\[
\mu(f) \sim \int_{\mathbb R^2_+} f(s+t, G_\h(s, t)) dsdt +O(\hbar^\infty).
\]
So the integral 
\[ 
\int_{\mathbb R^2_+} f(s+t,G_\h(s, t)) dsdt
\]  
is spectrally determined for any compactly supported function $f$. 
Making the change of variables $u=s+t, v=G_\h(s, t)$, the previous integral becomes 
\[ 
\int_{\mathbb R^2_+} f(u, v) (\frac{\partial G_\h}{\partial s} - \frac{\partial G_\h}{\partial t})^{-1} dudv .
\]  
Therefore, the Jacobian 
\[
\frac{\partial G_\h}{\partial s} - \frac{\partial G_\h}{\partial t}
\]
is spectrally determined. In other words, one can recover the function $G_\h$ 
(modulo $O(\h^\infty)$)
up to an additive factor of the form $g(s+t)$. 

To determine this additive factor,  we can fix a $\lambda$ and pick test functions $f$ that are
supported in shrinking neighborhoods of the line segment $s+t=\lambda$ 
inside the first quadrant. Using this 
one can determine the range of the function $G_\h$ restricted to the line segment $s+t=\lambda$. 
Therefore the unknown function $g$ is also spectrally determined. 
\end{proof}


\section{The local case}

The results we described in previous sections extend to functions whose reductions are not 
perfect Morse on each reduced space $S_\lambda^2$. In this section we will replace the perfect 
Morse assumption (\ref{globalHyp}) by: 
\[\aligned
&\mathrm{For\ all\ } \lambda \in I, \text{there exist }a_\lambda \text{ and }b_\lambda, \text{ depending smoothly on }\lambda, \text{ such that } 
\\
&\mathrm{(1)}  (H_2^\lambda)^{-1} (a_\lambda) \text{ is connected.} \\
& \mathrm{(2)}  H_2^\lambda \text{ has no critical point in the open set } U_\lambda :=(H_2^\lambda)^{-1}((a_\lambda, b_\lambda)).
\endaligned\]
Note that the two conditions imply that $(H_2^\lambda)^{-1}(t)$ is connected for all $a_\lambda < t<b_\lambda$.

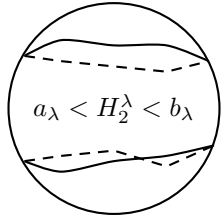
 \begin{center}
 \begin{tikzpicture}[scale=.7]

 \draw[thick, smooth] plot coordinates
{(-1.7,-1)(-1,-1.2)(0,-1)(1,-.9)(1.9, -.7)};
\draw[thick, dash pattern= on 4pt off 3pt]
plot coordinates
{(-1.7,-1)(-1,-.9)(0,-.8)(1,-1.1)(1.9, -.7)};

\node at (0, 0){\small $a_\lambda< H_2^\lambda<b_\lambda $};

 \draw[thick, smooth] plot coordinates
{(-1.7,1)(-1,1.3)(0,1.2)(1,1.2)(1.8, .9)};
\draw[thick, dash pattern= on 4pt off 3pt]
plot coordinates
{(-1.7,1)(-1,0.9)(0,0.8)(1,0.7)(1.8, .9)};

\draw [thick]  (0, 0) circle (57pt);
 
\end{tikzpicture} 
\end{center}  

We will let 
\[
M = \bigcup_{\lambda \in I} \pi_\lambda^{-1}\left( (H_2^\lambda)^{-1} ((a_\lambda, b_\lambda))\right),
\]
where $\pi_\lambda: S_\lambda^3 \to S_\lambda^2$ is the projection. 
Then $M \subset \mathbb R^4$ is an open symplectic submanifold, and the map 
\[
\Phi = (H_0,  H_2): M \to \mathbb R^2
\]
is a semitoric completely integral system without singular values. By the same argument as in section 2, one can convert this  semitoric system to a toric system on 
\[
\widetilde M = \cup_{\lambda \in I}\pi_\lambda^{-1}(X_3^{-1}(c_\lambda, d_\lambda)),
\]
where $c_\lambda$, $d_\lambda$ are determined by 
\[
\mathrm{Area}(H^\lambda_2<a_\lambda) = \mathrm{Area}(X_3<c_\lambda), \quad \mathrm{Area}(H^\lambda_2<b_\lambda) = \mathrm{Area}(X_3<d_\lambda).
\]
The same arguments as in section 3 imply 
\begin{theorem}
There exists a symplectomorphism $F: M \to \widetilde M$  that conjugates the hamiltonian 
$H=H_0+\hbar^2 H_2 +O(\hbar^3)$ on $M$ to the Birkhoff canonical form 
$H_0+\hbar^2 F_2(H_0^1, H_0^2)+O(\hbar^3)$ on $\widetilde M$. 
\end{theorem}
Similarly one can extend the remaining theorems in sections 3, 4, and 5 to this setting.

\section{Other examples}
In this section we briefly describe other settings to which our results or techniques can be extended 
easily.
\subsection{Perturbations of $\Delta_{S^2}$ by vector fields}  Consider an operator of the form
\[
\Delta + Q_\h,
\]
where $\Delta$ is the Laplacian on the round sphere of radius one $S^2$, and 
\begin{equation}\label{divfree}
Q_\h =  \frac{\h}{i}D,
\end{equation}
where $D$ is a divergence-free vector field on $S^2$
(more generally we can consider $Q_\h$ to be a zeroth order self-adjoint
semiclassical pseudo differential operator).   In this case the classical semi-toric system 
will consist of the symplectic
manifold
\begin{equation}\label{cots2}
M = T^*S^2\setminus\{0\},
\end{equation}
with its standard symplectic form, $\tilde{H}_0(x,\xi) = |\xi|$, and $H_1:M\to\bbR$ the averaged symbol of $Q_\h$.
The Hamilton flow of $\tilde{H}_0$ is $2\pi-$periodic, as it is geodesic flow re-parametrized by
arc length.

\newcommand{\n}{\vec{n}}
The reduced spaces are spheres, $S^2_\lambda$.  This can be seen as follows:  
An oriented geodesic $\gamma\subset S^2$ (an oriented great circle) is the intersection of $S^2$ with an 
oriented plane, which is characterized by its positive unit normal vector, $\n_\gamma$.  The bijective
correspondence $\gamma\to\n_\gamma$ shows that the reduced spaces are spheres.
The fiber dilations in $T^*M$ induce diffeomorphisms $S^2_1\to S^2_\lambda$ that simply 
rescale the symplectic form by a factor of $\lambda$.

Now we let $H_0 = \norm{\xi}^2 = \tilde{H}_0^2$ and consider a perturbation $H_0+\h H_1$.
By Moser averaging, the Hamiltonian system $H_0+\h H_1$ can be conjugated within $M$ to a 
system of the form $H_0+ \h W +\cdots$ where $W = H_1^{av}$.   At this point we present some details
on the averaging method:  If ``ave" denotes averaging of functions on $M$ with respect to the circle action
generated by $\tilde{H}_0$, the averaging method starts by finding a function $X:M\to\bbR$ that solves the
equation
\begin{equation}\label{ave1}
\{X , H_0\} = H_1^{ave} - H_1.
\end{equation}
But this equation is equivalent to $\{X , \tilde{H}_0\} = (H_1^{ave}-H_1)/2\tilde{H}_0$,  which is solvable because
$(H_1/\tilde{H}_0)^{ave} = H_1^{ave}/\tilde{H}_0$.  Then one has
\[
\exp(\h\Xi_X)^*(H_0 + \h H_1) = H_0 + \h H_1^{ave} + \cdots
\]
where the dots represent terms of order $O(\h^3)$ and higher.  Clearly the method can be continued indefinitely.

\medskip
We now concentrate on the case (\ref{divfree}).  We will show that there are many examples
of perturbations of this sort where the perfect Morse condition holds.
By homogeneity of the symbol of $D$,
under the aforementioned dilations, 
$S^2_1\to S^2_\lambda$,
the various reduced functions $W^\lambda: S^2_\lambda\to\bbR$ are scalar multiples of each other.
We now investigate the condition that these functions are perfect Morse.  For this we need to compute
the averaging procedure for symbols of divergence-free vector fields.

Let $\varpi$ be the area form of $S^2$, which we now think of as a symplectic form.  Since $S^2$
is simply connected, that $D$ is divergence free (i.e. $\calL_D\varpi = 0$) is equivalent to the existence
of a function $f:S^2\to\bbR$ such that $D$ is the Hamilton vector field of $f$ with respect to $\varpi$.
We will consider the composition
\begin{equation}\label{funRadon}
\begin{array}{rcc}
\calR:  C^\infty(S^2) & \to & C^\infty(S_1^2)\\
f & \mapsto & W^1
\end{array}
\end{equation}
mapping a ``hamiltonian" $f$ to the reduction at $\lambda = 1$ of the symbol of its Hamilton 
vector field.

By Schur's lemma $\calR$ maps each space of spherical harmonics into itself, and on each such
space it is a multiple of the identity.  A calculation of 
$\calR [(x+iy)^k]$ at one point yields the following:

\begin{lemma} (See also \cite{Ku})
The kernel of $\calR$ consists exactly of the even functions.  In addition, if $\calH_{2m+1}\subset C^\infty(S^2)$
denotes the space of spherical harmonics of degree $2m+1$ then $\calR$ maps this space into itself and the
restriction equals the identity times
\[
\calR|_{\calH_{2m+1}} = \frac{2m+1}{2\pi}\cdot \frac{(-1)^m}{4^m}\cdot {2m\choose m} \sim
(-1)^m\, \frac{\sqrt{m}}{\pi^{3/2}},
\]
the asymptotics being as $m\to\infty$.
\end{lemma}

In particular, if $f(x,y,z) = z$, then $\calR(f) = \frac{1}{2\pi}z$ is a perfect Morse function and any function
sufficiently close to $z$ in a suitable Sobolev norm will be perfect Morse as well.  Or
we can start with a perfect odd Morse function on $S^2_1$, and take for $D$ the Hamilton
field of its inverse image under $\calR$.  This shows that there are many examples where 
condition (\ref{globalHyp}) is satisfied.

We can apply our results to perturbations given by the hamiltonian fields of these functions,
to conclude that:

\begin{enumerate}
\item There is a two-variable function $\Phi$ such that $(H_0, \Phi(H_0, H_1))$ is the 
moment map of a torus action on $M$
\item There is a canonical transformation $ M\to M$ transforming $(H_0, \Phi(H_0, H_1))$
into the moment map of $(H_0, c\sigma_{\frac{1}{i}\partial z})$, where $\partial_z$ is the 
infinitesimal generator of the rotations of $S^2$ around the $z$ axis and $c$ a suitable
constant.
\end{enumerate}

On the quantum side the averaging method can also be implemented, 
using the exponential of the operator
\[
A = \sqrt{\Delta_{S^2}+\frac{1}{4}} - \frac{1}{2},
\]
which is $2\pi$ periodic and whose symbol is $H_0$, as the operator with respect to which the quantum
averaging takes place.  With minor modifications the existence
of the quantum BCF, and the proof that the quantum BCF can be recovered from the
semiclassical spectrum of $\Delta + Q_\h$, can be established as before.

\subsection{Perturbations of a magnetic laplacian on $S^2$ by potentials.}

In this example we take again $M=T^*S^2\setminus\{0\}$, but with the symplectic form
\[
\Omega = \omega + \pi^*\varpi,
\]
where $\omega$ is the natural symplectic structure on $T^*S^2$.  $\pi: T^*S^2\to S^2$
is the projection and as before $\varpi$ is the area form on $S^2$.  
$(M, \Omega, H_0)$, $H_0(x,p) = \frac{1}{2}\norm{p}^2$,  is the 
Hamilton system associated to a charged particle on $S^2$ under the influence of the 
magnetic field $\varpi$ (setting all physical constants equal to one).  
The trajectories of this system are geodesic circles, whose radii
depend on the energy $H_0$, as we will see.

In spherical coordinates (with $\phi$ the azimuthal angle and $\theta$ the polar angle
\footnote{The opposite convention is used in the physics literature.}),  one has:
\[
\Omega =d p_\theta\wedge d\theta +d p_\phi\wedge d\phi + \sin(\phi) d\phi\wedge d\theta
\]
and
\[
H_0 = \frac{1}{2}\left(\frac{p_\theta^2}{\sin^2(\phi)} + p_\phi^2
\right).
\]
A calculation shows that the equations of motion are:
\begin{equation}\label{sphLandau}
\left\{
\begin{array}{ccl}
\dot p_\theta & = & -\dot\phi\, \sin(\phi)\\
\dot p_\phi & = & \dot\theta\,\sin(\phi) + \frac{p_\theta^2\,\cos(\phi)}{\sin^3(\phi)}\\
\dot\theta & = & \frac{p_\theta}{\sin^2(\phi)}\\
\dot\phi & = & p_\phi .
\end{array}
\right.
\end{equation}
By rotational symmetry we only need to compute the trajectories with $\dot\phi = 0$.
These are geodesic circles centered at the north pole with geodesic radius a fixed
value of $\phi$, determined by an initial value of $p_\theta$ (and with $p_\phi = 0$).  
From the first and third equations in (\ref{sphLandau}) we see that
for those solutions $p_\theta$ is constant, and
\[
\theta(t) = t\frac{p_\theta}{\sin^2(\phi)} + \theta(0).
\]
It will follow that  a trajectory of this kind is periodic, with period
\begin{equation}\label{period1}
T = \pm 2\pi \frac{\sin^2(\phi)}{p_\theta}
\end{equation}
(the sign of $p_\theta$ being unknown at this time).
It also follows from (\ref{sphLandau}) that along the solutions we are considering
\[
0 = \dot p_\phi = \frac{p_\theta}{\sin(\phi)} + \frac{p_\theta^2\,\cos(\phi)}{\sin^3(\phi)},
\]
which implies that
\begin{equation}\label{mm}
p_\theta = -\frac{\sin^2(\phi)}{\cos(\phi)}.
\end{equation}
Combining (\ref{period1}), (\ref{mm}) we see that the period is
\[
T = 2\pi\cos(\phi).
\]
On the other hand,
combining (\ref{mm}) with the energy of the trajectory,
$
\lambda = \frac{1}{2}\,\frac{p_\theta^2}{\sin^2(\phi)},
$
we obtain 
\begin{equation}\label{radius}
\lambda = \frac{1}{2} \tan^2(\phi).
\end{equation}
Using these expressions for $T$ and $\lambda$ we
can  conclude:

\begin{proposition}
The trajectories of the hamiltonian system $(M, \Omega, H_0)$ are all periodic, and project
to geodesic circles on $S^2$.  Moreover, if  $\lambda$ is the energy of a trajectory, then its period
is equal to \footnote{In contrast with the analogous system on the plane, the period depends 
on the energy.}
\[
T(\lambda) = \frac{2\pi}{\sqrt{1+2\lambda}}, \quad \lambda\in(0,\infty).
\]
\end{proposition}
Let $\chi(\lambda) = \sqrt{1+2\lambda}$ so that $\chi'(\lambda) = T(\lambda)/2\pi$.  Then 
the Hamilton flow of $\tilde H_0=\chi(H_0)$ is $2\pi$ periodic, and the reduced spaces are
once again spheres.

We now consider the classical Hamiltonian
\begin{equation}\label{mag+V}
H(x,\xi) = H_0(x,\xi)+ \h V(x),
\end{equation}
with $V:S^2\to\bbR$ smooth.
Just as in the previous example, the averaging method
can be applied, using $\tilde H_0$ as the averaging hamiltonian. 
The method continues to work because the averaging hamiltonian is a function of 
$H_0$.  Therefore we can conjugate (\ref{mag+V}) to a hamiltonian of the form
\[
H_0 + \h V^{ave} + \cdots
\]
where all terms of the expansion Poisson commute with $H_0$.

\medskip
Our semi-toric system is now $(M,\tilde H_0, V^{ave})$.
The methods developed in the main body of this paper continue to apply whenever the reduction 
of $V^{ave}$ to the quotient spheres $S^2_\lambda$ is perfect Morse, either for all $\lambda$ or
for a range of $\lambda$ (the local case).
We now describe the reduction of $V^{ave}$ to $S^2_\lambda$.

Recall that the projections of the trajectories of $H_0$ onto
$S^2$ are geodesic circles whose radii $r$ are related to the energy $\lambda$ by the equation
\[
r = \tan^{-1}( \sqrt{2\lambda})
\]
(this is just equation (\ref{radius}) plus the observation that $\phi$ is 
the function ``geodesic distance from the north pole").  The mapping $V\mapsto $ reduction of $V^{ave}$
to $S^2_\lambda$ can be identified, up to a multiplicative constant, with the spherical mean transform
\[
M_r: C^{\infty}(S^2)\to C^\infty(S^2)
\]
where $M_r(V)(x)$ is the average of $V$ over the geodesic circle with center $x$ and radius $r$,
with respect to arc length.  (For a computation of this map in terms of the spherical harmonics decomposition
of $V$ see Lemma 7.2 of \cite{GU}.)  It is clear that if $V$ itself is perfect Morse then $M_r(V)$ will be
as well if $r$ is small enough.

\medskip
We briefly sketch a quantization of the previous system.  For each $k=1, 2,\ldots$, let
$\Delta_k$ be the Laplacian acting on sections of $L^{\otimes k}\to S^2$, where $L\to S^2$ is the hermitian
line bundle with the SO$(3)$ invariant connection whose curvature is $\varpi$.  Then the quantum hamiltonian
corresponding to the previous classical system is $\Delta_k + \frac{1}{k} V$, where $k = 1/\h$.  
The eigenvalues of $\Delta_k$ are (c.f. \cite{Ku})
\[
\left( j+\frac{k+1}{2}\right)^2 - \frac{k^2+1}{4}, j=0,1,\ldots ,
\]
with multiplicities $k+2j+1$.  The spectrum of  $\Delta_k + \frac{1}{k} V$ forms clusters around
these eigenvalues, of size $O(1/k)$.  Under a ``perfect Morse" assumption one can show that, as before:
(a) there exist both classical and quantum Birkhoff normal forms and (b) that the leading term
of the classical normal form and the complete quantum normal form are determined by the
spectra of the operators $\Delta_k + \frac{1}{k} V$, $k=1, 2,\ldots$.


\end{document}